\documentclass[11pt,reqno]{amsart}
\usepackage[T2A]{fontenc}
\usepackage[cp1251]{inputenc}
\usepackage[russian,english]{babel}

\usepackage{xcolor} 

\usepackage[pdftex,unicode,colorlinks,linkcolor=blue,
citecolor=red,bookmarksopen,pdfhighlight=/N]{hyperref}

\usepackage{indentfirst}
\usepackage{verbatim}
\usepackage{amsfonts,amssymb,mathrsfs,amscd}
\usepackage{amsmath,latexsym}
\usepackage{amsthm}
\usepackage{verbatim}
\usepackage{eucal}
\usepackage{graphicx}
\usepackage{enumerate}

\usepackage{amsbib}

\theoremstyle{plain} 
\newtheorem{theorem}{Теорема}

\newtheorem{lemma}{Лемма}[section]
\newtheorem{proposition}{Предложение}[section] 
\newtheorem{corollary}{Следствие}[section] 
\theoremstyle{definition}
\newtheorem{definition}{Определение}[section] 
\newtheorem{remark}{Замечание}[section]

\renewcommand{\leq}{\leqslant} 
\renewcommand{\geq}{\geqslant}
 
\newcommand{\rad}{\text{\tiny\rm rad}}

\newcommand{\RR}{\mathbb{R}} 
\newcommand{\CC}{\mathbb{C}}

\newcommand{\e}{\varepsilon}

\DeclareMathOperator{\mes}{mes}

\DeclareMathOperator{\clos}{clos}

\DeclareMathOperator{\dens}{dens} 
\DeclareMathOperator{\Meas}{Meas} 
 
\DeclareMathOperator{\Zero}{Zero} 
\DeclareMathOperator{\sbh}{sbh}

\DeclareMathOperator{\supp}{supp} 

\DeclareMathOperator{\type}{type}

\DeclareMathOperator{\rh}{\text{\rm \tiny rh}}
 \DeclareMathOperator{\lh}{\text{\rm \tiny lh}}
 
\DeclareMathOperator{\strip}{str} 
 
\DeclareMathOperator{\loc}{loc}

\DeclareMathOperator{\dd}{\,{\mathrm d\!}}
 
\renewcommand{\Re}{{\rm Re \,}}
\renewcommand{\Im}{{\rm Im \,}}

\begin{document}

%        Заголовок статьи.
\title{Рост субгармонических функций вдоль мнимой оси}

\author{А.\,Е.~Егорова, Б.\,Н.~Хабибуллин}

\selectlanguage{russian}
\maketitle

\section{Введение}\label{s10}

\subsection{Основная задача. Истоки}\label{subs1_1}  Пусть  $u\not\equiv-\infty$ и $M\not\equiv -\infty$ --- субгармонические 
 на \textit{комплексной плоскости\/} $\CC$ функции \textit{конечного типа\/} (при порядке $1$), что означает конечность как типа 
 \begin{equation}\label{typev}
\type[u]:=\limsup_{z\to \infty}\frac{u(z)}{|z|},
\end{equation}
так и типа $\type[M]<+\infty$,  с мерами Рисса соответственно $\nu_u:=\frac{1}{2\pi}\bigtriangleup\! u$ и 
$\mu_M:=\frac{1}{2\pi}\bigtriangleup\! M$, где  $\bigtriangleup$ --- \textit{оператор Лапласа,\/} действующий в смысле теории обобщённых функция
\cite{HK}, \cite{Rans}. 
%%; $\RR\subset\CC$ и $i\RR\subset \CC$ --- соответственно  \textit{вещественная %%ось\/} и \textit{мнимая ось.\/}
Предположим, что рост функции $u$ вдоль некоторой прямой $L\subset \CC$
%%, т.\,е. значения функции $u(z)$ при  $z\in L$, 
мажорируется функцией $M$ или, более общ\'о, некоторыми усреднениями функции $M$ по окружностям с центрами на $L$, к тому же с определёнными аддитивными добавками к $M$, а также не всюду на $L$, а вне некоторого исключительного  множества $E\subset L$. В таком случае естественно ожидать, что мера Рисса $\nu_u$ тоже должна в каком-то смысле мажорироваться мерой Рисса $\mu_M$
в увязке с радиусами усредняющих окружностей, с характеристиками аддитивных добавок и степенью малости исключительного множества $E$. 
Наша \textit{основная задача\/} --- дать количественные характеристики такого мажорирования меры $\nu_u$ мерой $\mu_M$ в терминах специальных <<логарифмических>>  характеристик/плотностей распределения мер $\nu_u$ и $\mu_M$. 
Полученная  в этом направлении  теорема \ref{prth1}, сформулированная 
ниже в подразделе \ref{main}, а также её вариации (предложение \ref{prQN},
следствие \ref{corJ}, теорема \ref{th2} из раздела \ref{thvar})
 --- результаты новые и  при специальном виде $u=\ln|f|$ и $M=\ln |g|$ в случае  \textit{целых функций экспоненциального типа\/} (пишем \textit{ц.ф.э.т.\/})  $f\neq 0$ и $g\neq 0$ с 
$\type\bigl[\ln |f|\bigr]<+\infty$ и $\type\bigl[\ln |g|\bigr]<+\infty$, когда в роли мер Рисса $\nu_u$ и $\mu_M$ выступают \textit{последовательности нулей,\/} или корней, $\Zero_f$
%%{\sf Z}_f=\{{\sf z}_k\}_{k=1,2,\dots}$   
и  %%${\sf W}_g=\{{\sf w}_k\}_{k=1,2,\dots}$ 
$\Zero_g$  соответственно функций $f$ и $g$, перенумерованные каким-либо образом с учётом кратности. Как раз для ц.ф.э.т. устанавливается заключительная теорема единственности  \ref{th3}. 
%%, перенумерованные с учётом кратности. 

В  постановке именно для ц.ф.э.т.  $f$ и $g$ версия нашей основной задачи  рассматривалась  в совместной работе П.~Мальявена и Л.\,А.~Рубела  \cite{MR}, в которой в качестве прямой $L$ выбиралась  {\it мнимая ось\/} $i\RR\subset \CC$, где $\RR\subset \CC$ --- {\it вещественная ось.\/} Такого выбора придерживаемся и мы. 
В \cite{MR}  для произвольной  ц.ф.э.т. $g$ с нулями
%% $\Zero_g$ которой 
в \textit{правой полуплоскости\/} $\CC_{\rh}:=\{z\in \CC \colon \Re z>0\}$ %%лежит
исключительно на {\it положительной полуоси\/} $\RR^+:=\{x\in \RR\colon x\geq 0\}$ было дано законченное описание всех  \textit{положительных\/} последовательностей точек  
${\sf Z}=\{{\sf z}_k\}_{k=1,2,\dots}
\subset \RR^+$, для каждой из  которых найдётся  своя  ц.ф.э.т. $f\neq 0$, обращающаяся в нуль на ${\sf Z}$  и удовлетворяющая ограничению $\ln |f(iy)|\leq \ln |g(iy)|$ при всех $y\in \RR$. 
%%${\sf W}_g=\{{\sf w}_k\}_{k=1,2,\dots}$, для которых  
Этой задаче посвящен один из основных разделов  монографии Л.\,А.Рубела в сотрудничестве с Дж.\,Э.~Коллиандром \cite[раздел 22]{RC} 1996 г.
В серии работ второго из авторов  
1988--91 гг. %%на несколько лет  раньше публикации монографии \cite{RC}  
все эти результаты были перенесены на произвольные 
{\it комплексные\/}  последовательности ${\sf Z} \subset \CC$
%%, отделенные какой-нибудь парой вертикальных углов от мнимой оси, 
с ограничением сверху вида $\ln |f(iy)|\leq M(iy)$ для всех  $y\in \RR$ через  
 \textit{специальную субгармоническую функцию-мажоранту\/}  $M$ вместо $\ln |g|$, а именно: со сколь угодно малым числом $\varepsilon >0$ 
для $M(z)=\e |z|$, $z\in \CC$, изначально в статье И. Ф. Красичкова-Терновского \cite[теоремы 8.3, 8.5, следствие 5.6]{Kra72} только для последовательностей $\sf Z$
вблизи $i\RR$, для любых $\sf Z\subset \CC$ в \cite[основная теорема]{Kha88} с  
дополнением  в \cite[основная теорема, теорема 1]{Kha01l}, а также в  гораздо более общей форме  с мажорантой вида $M(z)=\ln |g(z)|+ \e|z|$, где $g\neq 0$  --- ц.ф.э.т.,  в   \cite[теорема 1]{KhaD88} и в \cite[основная теорема]{Kha89}, или ещё более общ\'о и жёстко, --- с мажорантами вида $M=\ln |g|$ уже  с $\e=0$,  но для последовательностей ${\sf Z}\subset \CC$, отделенных какой-нибудь парой вертикальных углов от мнимой оси, --- в \cite[основная теорема]{kh91AA}.  Ситуация с    субгармонической функцией-мажорантой $M$ конечного типа при порядке $1$, гармонической  в паре вертикальных углов, содержащих $i\RR\setminus \{0\}$, достаточно детально  исследована  в диссертации второго из соавторов  \cite[гл.~II]{KhDD92}, но в научных журналах  последние  результаты с произвольной субгармонической мажорантой $M$ конечного типа не публиковались. 
Большинство отмеченных выше результатов изложено   в монографии-обзоре  \cite[3.2]{Khsur} с подробными комментариями. 

\subsection{Обозначения и определения}\label{subs1_2}
В данном подразделе приводится  всё, что использовано  для формулировки основной в статье теоремы \ref{prth1}. Так, $\sbh$ ---  множество всех \textit{субгармонические функции на\/} $\CC$, $\sbh_*:=\{u\in \sbh \colon u\not\equiv -\infty\}$.

Одним и тем же символом $0$ обозначаем, по контексту, число нуль, нулевую функцию, нулевую меру и т.\,п.; $\varnothing$ --- {\it пустое множество.\/} 
{\it Положительность\/} всюду понимается %%, в соответствии с контекстом, 
как $\geq 0$, а \textit{отрицательность} --- это $\leq 0$.  

$\Meas^+$ --- класс всех \textit{положительных борелевских мер на\/} $\CC$, 
$\mes$ --- \textit{линейная мера Лебега} на $\RR$.
%%% и  $\mes X:=\mes (X)$ для $X\subset \RR$.
$C(X)$ --- класс \textit{всех непрерывных функций\/} $f\colon X\to \RR$. 
$L_{\loc}^1(X)$  для  \textit{$\mes$-измеримого подмножества\/} $X\subset \RR$ --- класс всех \textit{локально $\mes$-ин\-т\-е\-г\-р\-и\-р\-у\-е\-м\-ых функций\/} со значениями в  $\RR_{\pm \infty}:=\{-\infty\}\cup \RR\cup \{+\infty\}$, где \textit{расширенная вещественная ось\/} $\RR_{\pm\infty}$ наделяется естественным порядком $-\infty\leq x \leq +\infty$, $x\in \RR_{\pm\infty}$, и порядковой топологией, или топологией конечной копактификации $\RR$ с двумя концами $\pm\infty$. Аналогично определяется $L_{\loc}^1(Y)$  для   $Y\subset i\RR$.

Пусть $X\subset \RR_{\pm\infty}$. Функция $f\colon X\to \RR_{\pm\infty}$ {\it возрастающая}, если для $x_1,x_2\in X$ из $x_1\leq x_2$ следует $f(x_1)\leq f(x_2)$; $f$ \textit{убывающая,\/} если $-f$ возрастающая. 

\textit{Интервал} --- связное подмножество в $\RR_{\pm \infty}$.  \textit{Интегралы Стилтьеса\/} по ограниченному в $\RR$ интервалу $I$ с концами $a:=\inf I<\sup I=:b$ \textit{по функциям ограниченной вариации\/} $m\colon \RR\to \RR$ на этом интервале $I$ обычно, если не оговорено противное, понимаем как 
 \begin{equation}\label{LS}
\int_a^b\dots \dd m :=\int_{(a,b]}\dots \dd m , \quad I=(a,b], \quad -\infty< a<b<+\infty. 
\end{equation}
$D(z,r):=\{z' \in \CC \colon |z'-z|<r\}$ --- {\it открытый,\/}  
$\overline{D}(z,r):=\{z' \in \CC \colon |z'-z|\leq r\}$ --- {\it замкнутый круги,\/} 
$\partial \overline{D}(z,r):=\overline{D}(z,r)\setminus {D}(z,r)$ --- {\it окружность с центром\/ $z\in \CC$ радиуса\/ $r\in \RR^+$}; $D(r):=D(0,r)$, 
 $\overline{D}(r):=\overline{D}(0,r)$, $\partial \overline{D}(r):=\partial \overline{D}(0,r)$.
Определим \textit{интегральные средние по  окружности $\partial \overline D(z, r)$ } от 
  $ v\colon \partial  \overline D(z,r)\to \RR_{\pm\infty}$: 
\begin{subequations}\label{df:MCB}
\begin{align}
\mathsf{C}_v(z, r)&:=:C(z, r;v):=\frac{1}{2\pi} \int_{0}^{2\pi}  v(z+re^{i\theta}) \dd \theta, \;\mathsf{C}_v(r):=\mathsf{C}_v(0, r),\tag{\ref{df:MCB}C}\label{df:MCBc}\\
\intertext{\textit{по кругу\/}$D(z,r)$ от фукции $v\colon D(z,r)\to \RR_{\pm \infty}$:}
\mathsf{B}_v(z,r)&:=:B(z,r;v):=\frac{2}{r^2}\int_{0}^{r}\mathsf{C}_v(z, t)t\dd t ,
\quad \mathsf{B}_v(r):=\mathsf{B}_v(0, r), \tag{\ref{df:MCB}B}\label{df:MCBb}\\
\intertext{а также верхнюю грань функции $v\colon  \partial \overline D(z,r)\to \RR_{\pm \infty}$ на окружности $\partial \overline D(z,r)$:}
 \mathsf{M}_v(z,r)&:=:M(z,r;v):=\sup_{z'\in \partial \overline D(z,r)}v(z'), 
\quad \mathsf{M}_v(r):=\mathsf{M}_v(0, r),  
\tag{\ref{df:MCB}M}\label{df:MCBm}
\end{align}
\end{subequations}
что  при $v\in \sbh %%\bigl( \overline D(z, r)\bigr)
$ совпадает с $\sup_{\overline D(z, r)}v$. Конечно же, в  \eqref{df:MCB} для  \eqref{df:MCBc} и \eqref{df:MCBb} подразумевается существование интегралов, что всегда имеет место для функций $v\in \sbh_*$ \cite[определение 2.6.7, теорема 2.6.8]{Rans}, \cite[2.7]{HK}, для которых 
\begin{equation}\label{BCM}
\mathsf{B}_v(z,r)\leq \mathsf{C}_v(z,r)\leq \mathsf{M}_v(z,r) \quad\text{при любых 
$z\in \CC$ и $r\in \RR^+$.}
\end{equation} 
   
Во всех упомянутых выше  в подразделе \ref{subs1_1} результатах из \cite{MR}--\cite{Khsur}  ключевую роль играли два объекта. Во-первых, это специальные интегралы по интервалам на вещественной или мнимой оси.
%% $(r,R]$ на вещественной оси $\RR$ или по $i\RR$.  
\begin{definition}\label{defJ} Для функции $v\in L^1_{\loc}(\RR)$ полагаем 
\begin{subequations}\label{J}
\begin{align}
J_{\RR}(r,R;v)&:=
\frac{1}{2\pi}\int_r^{R} \frac{v(x)+v(-x)}{x^2} \dd x, \quad 0<r<R<+\infty,
\tag{\ref{J}r}\label{fK}\\
\intertext{а для функции $v\in L^1_{\loc} (i\RR)$ ---}
J_{i\RR}(r,R;v)&:=\frac{1}{2\pi}
\int_r^{R} \frac{v(-iy)+v(iy)}{y^2} \dd y, \quad 0<r<R<+\infty.
\tag{\ref{J}i}\label{fK:abp+}
\end{align}
\end{subequations}
\end{definition}

%%Пусть $\Meas^+$ --- 
Второй объект --- это {\it логарифмические меры  интервалов для последовательностей  точек\/} ${\sf Z}=\{{\sf z}_k\}_{k=1,2,\dots}
\subset \CC$. Мы отождествляем каждую последовательность  $\sf Z$
со {\it считающей мерой\/} $n_{\sf Z}\in \Meas^+$, определяемой по правилу 
\begin{equation}\label{nZ}
n_{\sf Z}(S):=\sum_{{\sf z}_k\in S} 1\quad \text{для всех $S\subset \CC$,}
\end{equation}
и определим эти логарифмические меры сразу для произвольных мер из $\Meas^+$.
\begin{definition}\label{logD}
Для меры $\mu\in \Meas^+$
\begin{subequations}\label{df:lm}
\begin{align}
l_{\mu}^{\rh}(r, R)&:=\int_{\substack{r < | z|\leq R\\ \Re z >0}} \Re \frac{1}{ z} \dd \mu(z), \quad 0< r < R <+\infty ,
\tag{\ref{df:lm}r}\label{df:dDlm+}\\
l_{\mu}^{\lh}(r, R)&:=\int_{\substack{r< |z|\leq R\\ 
\Re z<0}}\Re \Bigl(-\frac{1}{ z}\Bigr) \dd \mu(z) ,  \quad 0< r < R < +\infty ,
\tag{\ref{df:lm}l}\label{df:dDlm-}\\
\intertext{---  {\it правая\/} и {\it левая логарифмические меры интервалов\/} $(r,R]\subset \RR^+$ для меры $\mu$ соответственно. Они 
%%в случае  {\it положительной меры\/} $\mu\in \Meas^+(\CC)$ 
порождают {\it логарифмическую субмеру интервалов}}
l_{\mu}(r, R)&:=\max \bigl\{ l_{\mu}^{\lh}(r, R), l_{\mu}^{\rh}(r,R)\bigr\}, \quad 0< r < R <+\infty. 
\tag{\ref{df:lm}m}\label{df:dDlLm}
\end{align}
\end{subequations}
\end{definition}

\subsection{Основной результат}\label{main}

\begin{theorem}\label{prth1}
Пусть для функций 
\begin{subequations}\label{uM}
\begin{align}
M&\in \sbh_*,\quad \type[M]\overset{\eqref{typev}}{<}+\infty,
\quad \mu:=\mu_M:=\frac{1}{2\pi}\bigtriangleup\! M\in \Meas^+,
\tag{\ref{uM}M}\label{{uM}M}\\
u&\in \sbh_*,\quad \type[u]<+\infty, \quad \nu:=\nu_u:=\frac{1}{2\pi}\bigtriangleup\! u \in \Meas^+,
\tag{\ref{uM}u}\label{{uM}u}\\
q_0&\colon \RR\to \RR^+\cup \{+\infty\},  \quad q_0\in L^1_{\loc}(\RR),
\tag{\ref{uM}$_0$}\label{{uM}0}
\\
 q&\colon \RR\to \RR^+, \quad q\in C(\RR), \quad
\limsup_{|y|\to +\infty}\frac{q(y)}{|y|}<1,
\tag{\ref{uM}q}\label{{uM}q}
\end{align}
\end{subequations}
и некоторого  $\mes$-измеримого подмножества $E\subset \RR^+$ с  
 \begin{subequations}\label{u<M}
\begin{align}
E^r:=E\cap [0,r],&\quad q_E(r):=\mes (E^r)\ln \frac{er}{\mes (E^r)}=:q_E(-r), 
\tag{\ref{u<M}E}\label{{u<M}E}\\
\intertext{имеют место неравенства}
u(iy)+u(-iy)&\leq \mathsf{C}_M\bigl(iy, q(y)\bigr)
+\mathsf{C}_M\bigl(-iy, q(-y)\bigr)
\tag{\ref{u<M}C}\label{{u<M}u}
\\
&+q_0(y)+q_0(-y)\quad\text{для  каждого числа  $y\in \RR^+\setminus E$}.
\notag
\end{align}
\end{subequations}
Тогда для любых  чисел $r_0>0$ и $N\in \RR^+$ найдётся число $C\in \RR^+$, 
для которого
\begin{equation}\label{rRuM}
\begin{split}
\max\bigl\{l_{\nu}(r,R), J_{i\RR}(r,R;u)\bigr\}\leq 
\min\bigl\{l_{\mu}^{\rh}(r,R), l_{\mu}^{\lh}(r,R), J_{i\RR}(r,R;M)\bigr\}
\\+CJ_{\RR} (r,R;q_0+q_E)+C I_N(r,R;q)+C
\quad\text{при всех $r_0\leq r<R<+\infty$,}
\end{split}
\end{equation}
где 
\begin{equation}\label{IN}
I_N(r,R;q):=\int_r^R t^N\sup_{s\geq t} \frac{q(s)+q(-s)}{s^{2+N}}\dd t.
\end{equation}
\end{theorem}
\begin{remark}\label{remr}  Для функций из \eqref{{uM}M} и \eqref{{uM}u} их меры Рисса $\mu$ и $\nu$ \textit{конечного типа,} или \textit{конечной верхней плотности,} при порядке $1$, что означает 
\begin{equation}\label{murad}
\type[\mu]:=\limsup_{r\to +\infty} \frac{\mu^{\rad}(r)}{r}<+\infty ,
\quad \mu^{\rad}(r):=\mu(D(0,r)), \quad \type[\nu]<+\infty. 
\end{equation}
\end{remark}
\begin{remark}
В силу неравенств \eqref{BCM} в правой части условия-неравенства \eqref{{u<M}u} 
средние по окружностям ${\sf C}_M$ из \eqref{df:MCBc} можно заменить на средние по кругам ${\sf B}_M$ из \eqref{df:MCBb}, но, вообще говоря,  нельзя заменить 
на верхние грани по окружностям или кругам ${\sf M}_M$ из \eqref{df:MCBm}.
\end{remark}

\section{Доказательство теоремы \ref{prth1}}
\setcounter{equation}{0} 
При доказательстве  теоремы \ref{prth1} можно рассматривать любое, но фиксированное значение $r_0>0$, поскольку  по определению\/ {\rm \ref{logD}} логарифмических мер интервалов \eqref{df:dDlm+}, \eqref{df:dDlm-}, \eqref{df:dDlLm} и определению \ref{defJ}
 интегралов\/  \eqref{fK:abp+}, \eqref{fK}  при изменении $r_0>0$ 
в  заключении \eqref{rRuM} теоремы \ref{prth1}
может разве что увеличиться постоянная $C\in \RR^+$. Поэтому в доказательстве теоремы\/ {\rm \ref{prth1}} мы при необходимости увеличиваем значение $r_0$, не оговаривая это специально.

%%%\begin{proof} 
Если в условии \eqref{{uM}q} рассматривать вместо функции $q$ функцию $q+1$, то её свойство из \eqref{{uM}q} 
%%и возможные дополнительные ограничения между 
%%\eqref{ub} и \eqref{qCg} 
не изменится,  условие-не\-р\-а\-в\-е\-н\-с\-т\-во \eqref{{u<M}u} сохранится, а на  заключение  \eqref{rRuM} это не повлияет. Поэтому
%%%, не умаляя общности,  
можем считать,  что $q\geq 1$ %%гладкая и 
 на $\RR$. Рассмотрим регулярную для задачи Дирихле \cite[теорема 2.11]{HK} область 
%%Тогда открытое множество   
\begin{equation}\label{Oq}
D:=D_{q}:=\bigl\{z\in \CC\colon -q(\Im z)< \Re z< q(\Im z) \bigr\}.
\end{equation}
 Для  $b\in \RR^+$ используем обозначение
\begin{equation}\label{strip}
\strip_b:=\bigl\{ z\in \CC \colon |\Im z|<b\bigr\},\quad
\overline{\strip}_b:=\bigl\{ z\in \CC \colon |\Im z|\leq b\bigr\}
\end{equation}
для соответственно \textit{открытой\/} и {\it замкнутой полос ширины $2b$ 
со средней линией $\RR$.} По предельному условию из \eqref{{uM}q} при достаточно большом $b>0$ часть  $D\setminus \overline\strip_{b}$ области $D$  состоит из двух односвязных областей, содержащихся внутри некоторой пары открытых вертикальных углов раствора $<\pi$ вида 
\begin{equation}\label{angle}
\angle (\alpha, \pi-\alpha):=\{z\in \CC\colon \alpha<\arg z<\pi-\alpha\}, \; \angle (-\pi+\alpha, -\alpha) \text{ с $\alpha \in (0,\pi/2)$.}   
\end{equation}
Можно произвести классическое выметание функции $M\in \sbh_*$ из области $D$ в два этапа. Сначала произведём выметание из этой пары односвязных  областей и выметенная функция останется субгармонической функцией конечного типа, поскольку она не превышает классического выметания рода $0$ функции $M$ из пары вертикальных углов \eqref{angle} \cite[6.2, теоремы 7, 8]{KhI}. На втором этапе осуществим классическое выметание из $D$ логарифмического потенциала рода $0$ появившейся меры  с компактным носителем в замыкании $D\cap \overline\strip_b$
\cite[гл. IV, \S~1]{L}, \cite[теорема 2.5.3.1]{Azarin} что может увеличить рост этого логарифмического потенциала лишь на величину  порядка $O\bigl(\ln |z|\bigr)$, $z\to \infty$. Таким образом, такая конструкция даёт  \textit{субгармоническую функцию}
%%\cite[3.8]{HK}, 
%%\cite[теорема 2.4.5]{Rans}
\begin{equation}\label{Uq}
M^{D}=\begin{cases}
M\quad \textit{на дополнении\/ $\CC\setminus D$},\\
\textit{гармоническое продолжение $M$ внутрь\/  $D$ на $D$}
\end{cases}
\end{equation}
\textit{конечного типа,\/} т.\,е. с $\type[M^D]<+\infty$ 
При этом по принципу максимума 
\begin{equation*}%%\label{CMiy}
{\mathsf C}_M\bigl(iy, q(y)\bigr)\leq M^D(iy)\quad\text{\it для всех\/ $y\in \RR$},
\quad M\leq M^D \text{\it  на }\CC,
\end{equation*}
а из условий-неравенств \eqref{{u<M}u} следует
\begin{equation}\label{uMD}
u(iy)+u(-iy)\leq M^D(iy)+M^D(-iy)+q_0(y)+q_0(-y)\text{ \it для всех\/ $y\in \RR\setminus E$.}
\end{equation}
Для меры $\nu\in \Meas^+$ \textit{сужение меры\/ $\nu$ на\/} $S\subset \CC$ обозначаем как $\nu\bigm|_S$.

Мера Рисса $\frac{1}{2\pi}\bigtriangleup\! M^D$ функции $M^D$ --- это сумма её сужений
\begin{equation}\label{mud}
\mu_{\infty}:=\frac{1}{2\pi}\bigtriangleup\! M^D\bigm|_{\CC\setminus \clos D}=
\mu\bigm|_{\CC\setminus \clos D}\leq \mu, \quad 
\mu_0:=\frac{1}{2\pi}\bigtriangleup\! M^D\bigm|_{\partial D}.
\end{equation} 
 Интегрирование неравенства \eqref{uMD} с множителем $\frac{1}{2\pi}$ в обозначении 
\begin{equation}\label{ErR}
E_r^R:=E^R\setminus E^r=E\cap (r,R]
\end{equation} 
даёт при всех значениях $r_0\leq r<R<+\infty$ неравенства
\begin{multline}\label{JuMD}
J_{i\RR}(r,R;u)\leq J_{i\RR}(r,R; M^D)\\
+\frac{1}{2\pi}\int_{E_r^R}\frac{u(iy)+u(-iy)-M^D(iy)-M^D(-iy)}{y^2}\dd y
+J_{\RR}(r,R;q_0).
\end{multline}
\begin{lemma}\label{lE} Пусть $r_0>0$. Для любой функции $u$ из \eqref{{uM}u} 
существует такое число $c_u\in \RR^+$, что для любого $\mes$-измеримого подмножества $E\subset \RR^+$ в обозначениях  \eqref{{u<M}E} и \eqref{ErR}  имеет место неравенство
\begin{equation}\label{Ju}
\int_{E_r^R}\frac{|u|(x)}{x^2}\dd x\leq c_u\int_r^R\frac{q_E(t)}{t^2}\dd t+c_u\quad\text{для всех $r_0\leq r<R<+\infty$,} 
\end{equation} 
где функция $q_E$ возрастающая и $q_E(r)\leq r$ при $r\in \RR_*^+$. 
\end{lemma}
\begin{proof}
Функция 
$$
(x,y)\longmapsto x\ln\frac{ey}{x}, \quad (x,y)\in (0,y] \times \RR^+,
$$ 
доопределённая нулём при $x=0$, возрастает по переменной $y\in \RR_*^+$, а также по  $x\in (0,y]$, достигая наибольшего значения при $x=y$, что даёт свойства функции $q_E$.  Далее $\boldsymbol{1}_E$ --- \textit{характеристическая  функция подмножества $E$.\/} 

Согласно \cite[теорема 8]{GrM} существуют постоянные $c',c\in \RR^+$, 
для которых  
\begin{equation}\label{Grint}
\int_{r_0}^x\boldsymbol{1}_{E}(t) |u|(t)\dd t\leq c'x\mes (E^x)\ln\frac{4x}{\mes (E^x)}
\overset{\eqref{{u<M}E}}{\leq} cq_E(x)x
\end{equation}
для всех $x\geq r_0$. Для левой части из \eqref{Ju} имеем 
\begin{multline*}
\int_{E_r^R}\frac{|u|(x)}{x^2}\dd x=
\int_r^R \boldsymbol{1}_E(x)\frac{|u|(x)}{x^2}\dd x=
\int_r^R\frac{1}{x^2}\dd \int_r^x \boldsymbol{1}_E(t)|u|(t)\dd t \dd x\\
=\frac{1}{R^2}\int_r^R  \boldsymbol{1}_E(t)|u|(t)\dd t+\int_r^R\int_r^x 
\boldsymbol{1}_E(t) |u|(t)\dd t \dd\,\Bigl(-\frac{1}{x^2}\Bigr)\\
\overset{\eqref{Grint}}{\leq} c\frac{q_E(R)}{R}+2c\int_r^R \frac{q_E(x)}{x^2}\dd x
\leq c+2c\int_r^R \frac{q_E(x)}{x^2}\dd x,
\end{multline*}
поскольку справедливо неравенство $q_E(R)\leq R$. Лемма \ref{lE} доказана. 
\end{proof}

Четырежды применяя лемму \ref{lE} к интегралу по множеству $E_r^R$
из правой части \eqref{JuMD}, для некоторого числа $c_1\in \RR^+$  
для всех $r_0\leq r<R<+\infty$ получаем
\begin{equation}\label{JDq}
J_{i\RR}(r,R;u)\leq J_{i\RR}(r,R; M^D)
+c_1J_{\RR}(r,R;q_E+q_0)+c_1.
\end{equation}

\begin{lemma}[{\cite[предложение 4.1, (4.19)]{KhII}\footnote{В формуле  (4.19) из  \cite{KhI}, к сожалению, был пропущен множитель $\frac{p}{2\pi}$ перед 
интегралом $J^{[p]}_{\alpha,\beta}(r,R;v)$, что, впрочем, не повлияло на результаты в \cite{KhI}. В настоящей статье соответствующий случаю $p:=1$, $\alpha :=-\pi/2$, $\beta:=\pi/2$ множитель $\frac{1}{2\pi}$ заранее прописан  в определении интеграла \eqref{fK:abp+} и  формула \eqref{Jl} корректна.}}]\label{lemJl} Пусть $r_0>0$. 
Для любой функции $u$ из \eqref{{uM}u} существует $c_u\in \RR^+$, для которого 
при всех $r_0\leq r<R<+\infty$
\begin{equation}\label{Jl}
\max \Bigl\{\bigl|J_{i\RR}(r,R;u)-l_{\nu}^{\rh}(r,R)\bigr|,
\bigl|J_{i\RR}(r,R;u)-l_{\nu}^{\lh}(r,R)\bigr|\Bigr\}
\leq c_u.
\end{equation}
\end{lemma}
Дважды применяя лемму \ref{lemJl} к функциям $u$ и $M^D$ 
в \eqref{JDq}, для некоторого числа $c_2\in \RR^+$ при всех $r_0\leq r<R<+\infty$
имеем неравенство
\begin{multline}\label{JDq+}
l_{\nu}^{\rh}(r,R)
\overset{\eqref{mud}}{\leq} l_{\mu_{\infty}+\mu_0}^{\rh}(r,R)%%J_{i\RR}(r,R; M^D)
+c_1J_{\RR}(r,R;q_E+q)+c_2\\
\overset{\eqref{mud}}{\leq}
 l_{\mu}^{\rh}(r,R)+ l_{\mu_0}^{\rh}(r,R)+c_1J_{\RR}(r,R;q_E+q_0)+c_2.
\end{multline}
Пусть  выбраны углы и зафиксировано число $b>0$  как в \eqref{Oq}--\eqref{angle}. Сужение  меры $\mu_0$ на замкнутую
 полосу $\overline\strip_b$ --- мера с компактным носителем и  для этого сужения  логарифмические меры и субмера  интервалов из \eqref{df:lm}  равномерно ограничены при всех $r_0\leq r<R<+\infty$. Следовательно, не умаляя общности, можем считать, что носитель меры $\mu_0$ содержится в паре углов \eqref{angle} и не пересекается с открытой полосой $\strip_b$. Обозначим  $\mu_0$-меру замкнутой полосы $\overline\strip_y$ ширины $2y\in \RR^+$ через 
\begin{equation}\label{mui}
\mu_0^i(y):=\mu_0(\overline\strip_y ) ,  \text{ где }
\mu_0^i(y)\leq Cy \quad \text{при всех }y\in \RR^+
\end{equation}
 для некоторой постоянной $C$, не зависящей от $y$, а также $\mu_0^i(y)=0$ при $y\in [0,b)$.

\begin{lemma}\label{lemD}
Пусть  $r_0>0$ и  $\mu_0\in \Meas^+$ --- мера  с носителем $\supp \mu_0$ в замыкании  $D\overset{\eqref{Oq}}{:=} D_{q}$, где функция $q$ из \eqref{{uM}q},  с  функцией $\mu_0^i\colon \RR^+\to \RR^+$  из  \eqref{mui}. 
Тогда для некоторого числа $C_0\in \RR^+$ 
\begin{equation}\label{lD}
l_{\mu_0}^{\rh}(r,R)\leq
\int_r^R  \frac{Q(y)}{y^2} \dd \mu_0^i(y)
%% C_0Q^*(r,R)
+C_0
\quad \text{при всех $r_0\leq r<R<+\infty$,}
\end{equation}
где 
\begin{equation}\label{Q}
Q(y):=q(y)+q(-y), \quad y\in \RR^+.  
\end{equation}
%%определена как в \eqref{ub}. 
%% с дополнениями в условиях \eqref{Q}--\eqref{qCg}. 
\end{lemma}
\begin{proof} Для  всех $b\leq r<R<+\infty$ имеем
\begin{multline}\label{qC}
l_{\mu_0}^{\rh}(r,R)\overset{\eqref{df:dDlm+}}{=}\int_{\substack{r < | z|\leq R\\ \Re z >0}}  \frac{\Re z}{|z|^2} \dd \mu_0(z)
\overset{\eqref{Oq}}{\leq} 
\int_{\substack{r < | z|\leq R\\ \Re z >0}}  \frac{q(\Im z)}{|\Im z|^2} \dd \mu_0(z) \\
\overset{\eqref{mui}}{\leq} \int_{r\sin \alpha}^{R} 
 \frac{q(y)+q(-y)}{y^2} \dd \mu_0^i(y)
\overset{\eqref{ub}}{\leq} \left(\int_{r\sin \alpha}^{r} +\int_{r}^{R} \right)
 \frac{Q(y)}{y^2} \dd \mu_0^i(y)\\
\overset{\eqref{mui}}{\leq} C\frac{Q(r)}{r\sin^2\alpha }
+\int_{r}^{R}  \frac{Q(y)}{y^2} \dd \mu_0^i(y)
\overset{\eqref{{uM}q}}{\leq} C_0+\int_r^R  \frac{Q(y)}{y^2} \dd \mu_0^i(y),
%% \quad b< r < R <+\infty ,
\end{multline} 
где $C_0\in \RR^+$ не зависит от $b\leq r<R<+\infty$. %%Лемма \ref{lemD} доказана. 
\end{proof}
По лемме \ref{lemD} из неравенства \eqref{JDq+} с постоянной $C_1\in \RR^+$ получаем
\begin{equation}\label{JDq+C}
l_{\nu}^{\rh}(r,R) \overset{\eqref{mud}}{\leq}  l_{\mu}^{\rh}(r,R)+ \int_r^R  \frac{Q(y)}{y^2} \dd m(y) +c_1J_{\RR}(r,R;q_E+q_0)+C_1,
\end{equation}
где положено $m(t):=\mu_0^i(t)$ и по построению для постоянной $C\in \RR^+$
\begin{equation}\label{m}
m(t)%%\mu_0^i(t)
\overset{\eqref{mui}}{\leq} Ct, \quad\text{при всех  $t\in  \RR^+$}, 
\quad \text{$m(t)=0$ при $t\in [0,b)\neq \varnothing$}.
\end{equation}
\begin{lemma}\label{lemm} Пусть возрастающая функция $m\colon \RR^+\to \RR^+$ удовлетворяет условиям \eqref{m}, 
а непрерывная функция 
\begin{equation}\label{QC}
Q\colon \RR^+\to \RR^+ \quad\text{такова, что } Q(t)=O(t) \text{ при } t\to +\infty.
\end{equation}  
Тогда для любого числа $N\in \RR^+$ найдётся число $C_2\in \RR^+$, для которого 
\begin{equation}\label{Qm}
 \int_r^R  \frac{Q(t)}{t^2} \dd m(t)\leq C_2\int_r^R t^N\sup_{s\geq t} \frac{Q(s)}{s^{2+N}}\dd t+C_2
\end{equation}
при всех $ b\leq  r < R <+\infty$, 
%%где 
%%\begin{equation}\label{Q}
%%Q^*(r,R):=\int_r^R t^N\sup_{s\geq t} \frac{Q(s)}{s^{2+N}}
%%\dd t \quad\text{для } Q(s):=q(s)+q(-s),
%%\end{equation}
%%функция $Q^*$ определена в \eqref{Q}, 
\end{lemma}
\begin{proof}  Для интеграла из левой части \eqref{Qm} имеем
\begin{equation*}
I:=\int_r^R  \frac{Q(t)}{t^2} \dd m(t)= 
\int_r^R  \frac{Q(t)}{t^2t^N} \dd \int_r^t s^N m(s) \leq
\int_r^R  \sup_{s\geq t}\frac{Q(s)}{s^{2+N}} \dd \int_r^t s^N m(s).  
\end{equation*}
Для  подынтегрального выражения в последнем интеграле, являющегося  
\textit{убывающей\/} функцией, ввиду \eqref{QC} существует число $C_3$, для которого
\begin{equation}\label{T}
T_N(t):=\sup_{s\geq t}\frac{Q(s)}{s^{2+N}}\leq C_3 \sup_{s\geq t}\frac{s}{s^{2+N}}
=C_3t^{-N-1} \text{ при всех $t\in [b,+\infty)$}.
\end{equation}
Интегрируя этот последний интеграл  по частям, получаем 
\begin{equation}\label{I1}
I\leq T_N(R) \int_r^R s^N \dd m(s) +\int_r^R \int_r^ts^N\dd m(s) \dd  \left(-T_N(t)\right).
\end{equation}  
С учётом  \eqref{m} оцениваем интеграл
\begin{equation*}
\int_r^ts^N\dd m(s)\leq m(t)t^N\leq Ct^{N+1},
\end{equation*}
что для правой части \eqref{I1} с учётом   \eqref{T} даёт
\begin{equation*}\label{I2}
I\leq CC_3+C\int_r^R t^{N+1} \dd  
\left(-T_N(t)\right)\leq CC_3+CC_3+C(N+1)\int_r^RT_N(t)t^N\dd t,
\end{equation*}
откуда  при $C_2:=\max\{2CC_3,C(N+1)\}$ получаем в точности \eqref{Qm}.
%%%Лемма \ref{lemm} доказана. 
\end{proof}
Из  \eqref{JDq+C}  по неравенству  \eqref{Qm} леммы \ref{lemm} 
для некоторой постоянной $C_4\in \RR^+$ %%следует 
\begin{equation}\label{JDqm}
l_{\nu}^{\rh}(r,R) \overset{\eqref{mud}}{\leq}  l_{\mu}^{\rh}(r,R)+ C_4
%%Q^*(r,R) 
\int_r^R t^N\sup_{s\geq t} \frac{Q(s)}{s^{2+N}}\dd t
+c_1J_{\RR}(r,R;q_E+q_0)+C_1
\end{equation}
при всех $+\infty>R>r\geq \max\{r_0,b\}$, где величину $\max\{r_0,b\}$ можно заменить на $r_0$, увеличивая при необходимости постоянную $C_1$, а $Q$ --- функция из \eqref{Q}. 
\begin{lemma}\label{lemlJ}
Пусть $r_0>0$. Для любой функции $u$ из \eqref{{uM}u} существует $C_u\in \RR^+$, для которого при всех $r_0\leq r<R<+\infty$ имеем неравенство 
\begin{equation*}
\max\bigl\{l_{\nu}(r,R), J_{i\RR}(r,R;u)\bigr\} \leq
\min\bigl\{l_{\nu}^{\rh}(r,R), l_{\nu}^{\lh}(r,R), J_{i\RR}(r,R;u)\bigr\}+C_u.
\end{equation*}
\end{lemma}
Доказательство леммы \ref{lemlJ} сразу следует из леммы   \ref{lemJl}.

Применяя  лемму \ref{lemlJ} к функциям $u$ и $M$ с мерами Рисса соответственно  $\nu$ и $\mu$,  из \eqref{JDqm} получаем \eqref{rRuM}, что завершает доказательство теоремы \ref{prth1}.
%%%\end{proof}

\section{Вариации заключения \eqref{rRuM} теоремы \ref{prth1}}\label{thvar} 
\setcounter{equation}{0} 

\subsection{Некоторые упрощения} Несколько громоздко  выглядящие интегралы  
\eqref{IN}, участвующие в правой части заключительной оценки \eqref{rRuM} теоремы \ref{prth1}, при очень незначительных дополнительных ограничениях на функцию 
\begin{equation}\label{ub}
Q_N(s):= \frac{Q(s)}{s^{2+N}}, \quad\text{где  $Q(s)\overset{\eqref{Q}}{:=}q(s)+q(-s)$,} 
\quad s \in \RR^+,
\end{equation} 
можно включить в наш стандартный интеграл $J_{\RR}$ вида \eqref{fK}, 
уже участвующий в правой части оценки \eqref{rRuM}. 

\begin{proposition}\label{prQN}
Пусть выполнено одно из следующих двух условий:
\begin{enumerate}[{\rm (i)}]
\item\label{Qi} функция $Q_N$ из \eqref{ub} убывающая на некотором интервале  $(A,+\infty)\neq \varnothing$; 
\item\label{Qii} функция $Q$  из \eqref{Q} непрерывно  дифференцируемая  на некотором интервале  $(A,+\infty)\neq \varnothing$  со свойством  
\begin{equation}\label{qCg}
\limsup_{y\to +\infty} \frac{yQ'(y)}{Q(y)}<+\infty.
\end{equation}
\end{enumerate}
Тогда найдутся числа $r_0>0$, $N\in \RR^+$ и  $C\in \RR^+$, для  которых 
\begin{equation}\label{Qbnt}
I_N(r,R;q)\overset{\eqref{IN}}{:=} \int_r^R t^N\sup_{s\geq t}Q_N(s)\dd t\leq C\int_r^R\frac{Q(t)}{t^2}\dd t
\overset{\eqref{fK}}{=}CJ_{\RR}(r,R;q).
\end{equation}
В частности, последний интеграл в \eqref{rRuM} можно заменить на $J_{\RR}(r,R; q)$, а  заключительную оценку  \eqref{rRuM} в теореме\/ {\rm \ref{prth1}} можно записать как 
\begin{multline}\label{rRuM-}
\max\bigl\{l_{\nu}(r,R), J_{i\RR}(r,R;u)\bigr\}\leq 
\min\bigl\{l_{\mu}^{\rh}(r,R), l_{\mu}^{\lh}(r,R), J_{i\RR}(r,R;M)\bigr\}\\
+CJ_{\RR} (r,R;q_0+q_E+q)+C
\quad\text{при всех $r_0\leq r<R<+\infty$,}
\end{multline}
\end{proposition}
\begin{proof}
 Если  функция   \eqref{ub} убывающая на $(A,+\infty)$, то, очевидно,
\begin{equation*}
 t^N\sup_{s\geq t}Q_N(s)\overset{\eqref{ub}}{=}t^N\frac{Q(t)}{t^{2+N}}=\frac{Q(t)}{t^2} 
\end{equation*}
при всех $t>A$ и при выборе $r_0>A$ получаем в точности \eqref{Qbnt}. 

Для непрерывно дифференцируемой функции $Q$, удовлетворяющей условию  \eqref{qCg}, найдётся постоянная $C_4$, с которой 
$sQ'(s)\leq C_5Q(s)$ при всех $s\geq b$. Производная функции из  \eqref{ub}
\begin{equation*}
Q_N'(s)=\frac{Q'(s)s-(2+N)Q(s)}{s^{3+N}}\leq \frac{C_5Q(s)-(2+N)Q(s)}{s^{3+N}},
\end{equation*}
 при выборе $N\geq C_5-2$ отрицательна и функция $Q_N$ из \eqref{ub} убывающая.
\end{proof}

Из предложения \ref{prQN} сразу получаем следующее очевидное 
\begin{corollary}\label{corJ}
Если в условиях  \eqref{Qi} или \eqref{Qii} предложения\/ {\rm \ref{prQN}} для  некоторого числа  $r_0>0$ имеет место соотношение 
\begin{equation*}
\sup_{r_0\leq r<R<+\infty}J_{\RR} (r,R; q_0+q_E+q)<+\infty, 
\end{equation*}
то заключительную оценку  \eqref{rRuM} в теореме\/ {\rm \ref{prth1}} можно записать как 
\begin{equation*}%%\label{rRuM-}
\max\bigl\{l_{\nu}(r,R), J_{i\RR}(r,R;u)\bigr\}\overset{\eqref{rRuM-}}{\leq} 
\min\bigl\{l_{\mu}^{\rh}(r,R), l_{\mu}^{\lh}(r,R), J_{i\RR}(r,R;M)\bigr\}
+C
\end{equation*}
при всех $r_0\leq r<R<+\infty$.
\end{corollary}
При известной асимптотике функций $q_0$, $q_E$, $q$ при приближении к 
$\pm\infty$ также можно упростить  заключительную оценку \eqref{rRuM} теоремы \ref{prth1}. Вариант ---
\begin{proposition}\label{pr2} Пусть для функции $p\colon \RR^+\to \RR^+$, ограниченной и интегрируемой по Риману на каждом ограниченном интервале $I\subset \RR^+$,
\begin{equation}\label{P0}
\limsup_{t\to +\infty} \frac{P(t)}{t}=0.
\end{equation}
Тогда для любого числа  $r_0>0$ найдётся убывающая функция $d\colon \RR^+\to \RR^+$, для которой 
\begin{equation}\label{d}
\lim_{R\to +\infty} d(R)=0 ,\quad
 \int_r^R\frac{P(t)}{t^2}\dd t\leq  d(R)\ln \frac{R}{r} \quad 
\text{при всех $r_0\leq r<R<+\infty$.}
\end{equation}
В частности, если для функций $q_0,q,q_E$ из  \eqref{{uM}0}, \eqref{{uM}q}, \eqref{{u<M}E} выполнены условия
\begin{equation*}%%\label{}
\limsup_{|y|\to +\infty}\frac{q_0(y)+q(y)+q_E(y)}{|y|}=0, \quad \sup_{[-R,R]}q_0<+\infty
\quad\text{ для любого $R\in\RR^+$},
\end{equation*}
а также функция $q_0$ локально интегрируема по Риману, то заключительную оценку  \eqref{rRuM} в теореме\/ {\rm \ref{prth1}} можно записать как 
\begin{multline*}%%\label{rRuM-}
\max\bigl\{l_{\nu}(r,R), J_{i\RR}(r,R;u)\bigr\}\overset{\eqref{rRuM-}}{\leq} 
\min\bigl\{l_{\mu}^{\rh}(r,R), l_{\mu}^{\lh}(r,R), J_{i\RR}(r,R;M)\bigr\}\\
+d(R)\ln \frac{R}{r}+C\quad\text{при всех $r_0\leq r<R<+\infty$.},
\end{multline*}
где $d$ --- убывающая функция, для которой $d(R)=o(1)$ при $R\to +\infty$.
%%из \eqref{d}.
\end{proposition}
\begin{proof} Сначала перейдём к {\it убывающей\/} ввиду  \eqref{P0} функции 
\begin{equation}\label{pP}
p(t):=\sup_{s\geq t}\frac{P(s)}{s}\geq \frac{P(t)}{t}, \; t\in \RR^+ ;
\quad \int_r^R\frac{p(t)}{t} \dd t \geq  \int_r^R\frac{P(t)}{t^2} \dd t .
\end{equation}
При каждом фиксированном числе $R>0$ положим 
\begin{multline}\label{dp}
d(R):=\sup_{r_0\leq r<R} \frac{1}{\ln (R/r)}\int_r^R\frac{p(t)}{t} \dd t
\geq \frac{1}{\ln (R/r)}\int_r^R\frac{p(t)}{t} \dd t\\
\overset{\eqref{pP}}{\geq}  \frac{1}{\ln (R/r)}\int_r^R\frac{P(t)}{t^2} \dd t
\quad\text{при всех значениях $r_0\leq r<R<+\infty$}, 
\end{multline}
что даёт последнее соотношение-неравенство из \eqref{d}.  

Точная верхняя грань в \eqref{dp} берётся от функции с частными производными
\begin{equation*}
\begin{split}
\frac{\partial}{\partial r}\frac{1}{\ln (R/r)}\int_r^R\frac{p(t)}{t} \dd t=
\frac{1}{r\ln^2(R/r)}\int_r^R\frac{p(t)-p(r)}{t}\dd t\leq 0 \quad\text{ при
$r_0\leq r<R$},\\
\frac{\partial}{\partial R}\frac{1}{\ln (R/r)}\int_r^R\frac{p(t)}{t} \dd t=
\frac{1}{R\ln^2(R/r)}\int_r^R\frac{p(R)-p(t)}{t}\dd t\leq 0 \quad\text{ при
$r_0\leq r<R$},
\end{split}
\end{equation*}
откуда эта функция убывает по $r<R$ и по $R>r$. Следовательно, 
\begin{equation}\label{dR}
d(R)\equiv \frac{1}{\ln (R/{r_0})}\int_{r_0}^R\frac{p(t)}{t} \dd t \quad\text{--- убывающая функция на $[r_0,+\infty)$.}
\end{equation}
Пусть выбрано число $a>0$ и $p(t)\leq a$ при $t\geq R_a$. Тогда из  \eqref{dR} при $R>R_a\geq r_0$
 \begin{equation*}
d(R)= \frac{1}{\ln (R/{r_0})}\left(\int_{r_0}^{R_a}+\int_{R_a}^R\right)\frac{p(t)}{t} \dd t 
\leq \frac{\sup_{[r_0,R_a]}p \ln(R_a/r_0)}{\ln (R/{r_0})}
+a\ln\frac{R}{R_a}ю
\end{equation*}
Отсюда  $\limsup_{R\to +\infty} d(R)\leq a$, что в силу произвола в выборе 
числа $a>0$ даёт соотношение $d(R)=o(1)$ при $R\to +\infty$.
\end{proof}

\subsection{Логарифмические меры и субмеры интервалов}

Понятиям логарифмических мер интервалов \eqref{df:dDlm+}, \eqref{df:dDlm-} и логарифмической субмеры интервалов  
\eqref{df:dDlLm}  из определения \ref{logD} можно придать и иную форму. 

\begin{definition}
Пусть $\mu\in \Meas^+$.  Введем в рассмотрение \textit{считающую  функцию меры\/  $\mu$ с $2\pi$-периодической борелевской  функцией-весом $k\colon \RR \to \RR^+$\/}:
 \begin{equation}\label{c)}
\mu(r;k):=\int_{\overline D(r)} k (\arg z)\dd \mu(z), 
\end{equation}
При $k\equiv 1$, очевидно, $\mu(r;1)\overset{\eqref{murad}}{\equiv} \mu^{\rad}(r)$ для $r\in \RR^+$.
\end{definition}

В частных случаях $k=\cos^{\pm}$, т.\,е.
\begin{equation}\label{kcos}
k(\theta):=\cos^+\theta:=\max\{0,\cos \theta\}, \;
 k(\theta):=\cos^-\theta:=\max\{0,-\cos\theta\}, \quad \theta\in \RR,
\end{equation}
 из определений \eqref{df:lm} в обозначениях \eqref{c)}--\eqref{kcos} 
 при %%любых значениях 
 $0< r < R <+\infty$  интегрированием по частям получаем 
\begin{subequations}\label{l_mu}
\begin{align}
l_{\mu}^{\rh}(r, R)&\overset{\eqref{df:dDlm+}}{=}
\int_r^R \frac{\dd \mu(t;\cos^+)}{t}
\tag{\ref{l_mu}r}\label{l_mu_m+}\\
&=\frac{\mu (R;\cos^+)}{R}-\frac{\mu (r;\cos^+)}{r}
+\int_r^R\frac{\mu (t;\cos^+)}{t^2} \dd t, 
\notag
\\
l_{\mu}^{\lh}(r, R)&\overset{\eqref{df:dDlm-}}{=}
\int_r^R \frac{1}{t} \dd \mu(t;\cos^-)
\tag{\ref{l_mu}l}\label{l_mu_m-}\\
&=\frac{\mu (R;\cos^-)}{R}-\frac{\mu (r;\cos^-)}{r}
+\int_r^R\frac{\mu (t;\cos^-)}{t^2} \dd t.  
\notag
\end{align}
\end{subequations}
Положим 
\begin{subequations}\label{l_mu-}
\begin{align}
\breve l_{\mu}^{\rh}(r, R)&\overset{\eqref{l_mu_m+}}{:=}\int_r^R\frac{\mu (t;\cos^+)}{t^2} \dd t, \quad 0< r < R < +\infty,
\tag{\ref{l_mu-}r}\label{l_mu_m-+}
\\
\breve l_{\mu}^{\lh}(r, R)&\overset{\eqref{l_mu_m-}}{:=}\int_r^R\frac{\mu (t;\cos^-)}{t^2} \dd t,\quad 0< r < R < +\infty,  
\tag{\ref{l_mu-}l}\label{l_mu_m--}
\\
%%\intertext{а для \textit{положительной меры\/} $\mu\in \Meas^+(\CC)$ ещё и}
\breve l_{\mu}(r, R)&\overset{\eqref{df:dDlLm}}{:=}\max \{ \breve l_{\mu}^{\lh}(r, R), \breve l_{\mu}^{\rh}(r, R)\}, \quad 0< r < R < +\infty .
\tag{\ref{l_mu-}m}\label{ml}
\end{align}
\end{subequations}
%%Из определений логарифмических функций/мер интервалов сразу следует 

\begin{proposition}\label{pr:lm}
Пусть $\mu\in \Meas^+$  --- мера  конечной верхней плотности,
или конечного типа ${\type}[\mu]\overset{\eqref{murad}}{<}+\infty$,  число $r_0>0$. Тогда 
\begin{equation}\label{ml+-}
\begin{cases}
\bigl|\breve l_{\mu}^{\rh}(r, R)-l_{\mu}^{\rh}(r, R)\bigr|=O(1)\\
\bigl|\breve l_{\mu}^{\lh}(r, R)-l_{\mu}^{\lh}(r, R)\bigr|=O(1) \\
\bigl|\breve l_{\mu}(r,R) - l_{\mu}(r,R)\bigr |=O(1)
\end{cases}
\text{для всех  $r_0\leq r<R<+\infty$.}
\end{equation}
Кроме того, для  любых фиксированных чисел  $a\in (0,1]$, $b\in [1,+\infty)$  
\begin{equation}\label{ml+ab}
\begin{cases}
\bigl| l_{\mu}^{\rh}(r, R)-l_{\mu}^{\rh}(ar,bR)\bigr|=O(1)\\
\bigl| l_{\mu}^{\lh}(r, R)-l_{\mu}^{\lh}(ar, bR)\bigr|=O(1) \\
\bigl|l_{\mu}(r,R) - l_{\mu}(ar,bR)\bigr |=O(1)
%%\text{ при $\mu\in\Meas^+(\CC)$}
\end{cases}
\text{для всех  $r_0\leq r<R<+\infty$.}
\end{equation}
\end{proposition} 
\begin{proof} Соотношения \eqref{ml+-} следуют из   \eqref{l_mu}, так как  для $r_0>0$ и меры $\mu$ конечной верхней плотности
\begin{equation*}%%\label{key}
\left|\frac{\mu (R;\cos^{\pm})}{R}-\frac{\mu (r;\cos^{\pm})}{r}\right|
\leq \frac{|\mu|^{\rad}(R)}{R}+\frac{|\mu|^{\rad}(r)}{r}=O(1)
\quad\text{при $r_0\leq r<R<+\infty$.}
\end{equation*}
 Соотношения  \eqref{ml+ab} 
следуют из \eqref{ml+-} и  из 
\begin{equation*}
\bigl|\breve l_\mu^{\rh}(ar,bR)-\breve l_\mu^{\rh}(r,R)\bigr|\leq
\sup_{ar<t\leq r}\frac{|\mu|^{\rad}(t)}{t}\ln\frac{1}{a}+
\sup_{R<t\leq bR}\frac{|\mu|^{\rad}(t)}{t}\ln b, 
\end{equation*}
где в левой части $\breve l^{\rh}$ можно заменить на  $\breve l^{\lh}$.
\end{proof}
\begin{remark}
По предложению\/ {\rm \ref{pr:lm}} согласно соотношениям  \eqref{ml+-}
и замечанию\/ {\rm \ref{remr}}  различные логарифмические меры  и субмеры интервалов\/ \eqref{df:dDlm+}, \eqref{df:dDlm-}, \eqref{df:dDlLm} из определения\/ {\rm \ref{logD}}, начинающиеся с $l$, в заключении  \eqref{rRuM} теоремы\/ {\rm \ref{prth1}} можно заменить на  соответствующие им по \eqref{ml+-} \textit{функции  интервалов\/} \eqref{l_mu_m-+}, \eqref{l_mu_m--}, \eqref{ml},
начинающиеся с  $\breve l$, для которых используем ту же терминологию. Далее эти две эквивалентные с точностью до аддитивной постоянной формы логарифмических мер и субмер интервалов для мер конечной верхней плотности можем не различать и обозначать как в определении\/ {\rm \ref{df:lm}}
без верхнего математического акцента\;  $\breve{}$ над $l$.
\end{remark}

\begin{definition}[{\rm \cite{KKh00}, развитие \cite[определения 3.4, 3.5]{MR}}]\label{drflsm} Пусть $0<r_0\in \RR^+$, $l$ --- функция интервалов $(r,R]\subset r_0+\RR^+$%% [r_0,+\infty)$ 
со значениями в $\RR_{\pm\infty}$, $l(r,R):=l\bigl((r,R]\bigr)$, для которой определим четыре логарифмические блок-плотности:
\begin{subequations}\label{dens}
\begin{align}
\ln\text{\!-}\overline\dens(l)&:=\limsup_{a\to +\infty} \frac{1}{\ln a} \limsup_{r\to+\infty} l(r,ar); 
\tag{\ref{dens}$^-$}\label{{dens}barl}\\
\ln\text{\!-\!}\underline{\dens}(l)&:=\liminf_{a\to +\infty} \frac{1}{\ln a} \limsup_{r\to+\infty} l(r,ar); 
\tag{\ref{dens}$_-$}\label{{dens}l}\\
\ln\text{\!-\!}\dens_{\inf}(l)&:=\inf_{a>1} \frac{1}{\ln a} \limsup_{r\to+\infty} l(r,ar); 
\tag{\ref{dens}i}\label{{dens}infl}\\
\ln\text{\!-\!}\dens_{\rm b}(l)&:=\inf
\left\{b\in \RR^+\colon \sup_{r_0\leq r<R<+\infty}\left(l(r,R)-b\ln\frac{R}{r}\right) <+\infty\right\}.
%%{a>1} \frac{1}{\ln a} \limsup_{r\to+\infty} l(r,ar); 
\tag{\ref{dens}b}\label{{dens}bl}
\end{align}
\end{subequations}
Функцию интервалов $l\geq 0$ называем логарифмической субмерой интервалов (вблизи $+\infty$), если для некоторого числа $r_0>0$ выполнены два условия:
\begin{enumerate}[{\rm [{\bf l}1]}]
\item\label{l1} $\sup_{r\geq r_0} l(r,2r)<+\infty$ {\rm (логарифмический рост)};
\item\label{l2} $l(r_1, r_3)\leq l(r_1, r_2)+l(r_2, r_3)$ для всех\/ $r_0\leq r_1<r_2<r_3$ {\rm (субаддитивность)}. 
\end{enumerate}
Если в первом неравенстве из\/ {\rm [{\bf l}\ref{l2}]} для некоторого числа $r_0>0$ знак $\leq$ можно заменить на знак $=$ для любых\/ $r_0\leq r_1<r_2<r_3$ {\rm (аддитивность),} то $l$ --- логарифмическая мера интервалов (вблизи $+\infty$).
\end{definition}
Из  определения \ref{drflsm} легко следуют
\begin{proposition}
Если $l_1$ и $l_2$ --- логарифмические субмеры интервалов, %% вблизи $+\infty$, 
то  $l_1+l_2$ и  $\max\{l_1,l_2 \}$ --- логарифмические  субмеры интервалов.
%% вблизи $+\infty$.
\end{proposition} 
\begin{proposition}\label{pr5} Для функций $q_0$ из \eqref{{uM}0} при условии 
\begin{equation}\label{q0}
\limsup_{|y|\to +\infty} \frac{q_0(y)}{|y|}<+\infty,
\end{equation} 
а также функций $q$ из \eqref{{uM}q} и $q_E$ из \eqref{{u<M}E} 
 интегралы $J_{\RR}(r,R;q_0)$, $J_{\RR}(r,R;q)$, $J_{\RR}(r,R;q_E)$ --- логарифмические меры интервалов $(r,R]\subset \RR^+$. Если $\mu\in \Meas^+$ --- мера конечной верхней плотности, то  $l_{\mu}^{\rh}$ и $l_{\mu}^{\lh}$ из \eqref{df:dDlm+} и  \eqref{df:dDlm-}, а также $\breve l_{\mu}^{\rh}$ и $\breve l_{\mu}^{\lh}$ из \eqref{l_mu_m-+} и \eqref{l_mu_m--} --- логарифмические меры интервалов,
%% вблизи $+\infty$, 
а $l_{\mu}$ из \eqref{df:dDlLm} и $\breve l_{\mu}$ из \eqref{ml} --- логарифмические субмеры интервалов. 
%% вблизи $+\infty$.
\end{proposition}

\begin{proposition}[{\rm \cite[теорема 1]{KKh00}}]\label{clb}
Для логарифмической субмеры интервалов $l\geq 0$ все четыре логарифмические блок-плотности из  \eqref{dens} конечны и совпадают, а $\limsup\limits_{a\to +\infty}$ в  \eqref{{dens}barl} и $\liminf\limits_{a\to +\infty}$ в \eqref{{dens}l} можно заменить на  предел $\lim\limits_{a\to +\infty}$. Далее для логарифмической субмеры интервалов $l\geq 0$  все четыре логарифмические блок-плотности из  \eqref{dens} обозначаем единообразно как $\ln\text{\!-\!}\dens (l)$.
\end{proposition} 
В  терминах логарифмической блок-плотности $\ln\text{\!-\!}\dens$ имеет место  

\begin{theorem}\label{th2} Пусть выполнены условия \eqref{uM} и \eqref{u<M} 
теоремы\/ {\rm \ref{prth1},} условие \eqref{q0}, для функции 
$Q(y):=q(y)+q(-y)$ из \eqref{Q},\eqref{ub} 
выполнено одно из условий  \eqref{Qi} или \eqref{Qii}, т.\,е. \eqref{qCg}, предложения\/ {\rm \ref{prQN}}, а также 
\begin{equation}\label{lJ}
\ln\text{\!-\!}\dens(J_{\RR}\bigl(\cdot,\cdot, q_0+q+q_E)\bigr)=0
\end{equation}
 Тогда 
\begin{equation}\label{D}
\ln\text{\!-\!}\dens (l_{\nu})\leq
 \min \bigl\{\ln\text{\!-\!}\dens(l^{\rh}_{\mu}),
\ln\text{\!-\!}\dens(l^{\lh}_{\mu})\bigr\}\leq 
\ln\text{\!-\!}\dens(l_{\mu}).
\end{equation}
\end{theorem}
\begin{proof} Заключение  \eqref{rRuM} теоремы \ref{prth1} при дополнительных 
условиях \eqref{Qi} или \eqref{Qii}, т.\,е. \eqref{qCg}, переходит в заключение \eqref{rRuM-} предложения \ref{prQN}, которое, в частности, для любого $a>1$ можно записать как 
\begin{equation*}
l_{\nu}(r,ar)\leq 
\min \bigl\{l_{\mu}^{\rh}(r,ar), l_{\mu}^{\lh}(r,ar)\bigr\} +CJ_{\RR} (r,ar;q_0+q_E+q)+C
\quad\text{при всех $r\geq r_0$.}
\end{equation*}
Отсюда, устремив $r$ к $+\infty$, получаем 
\begin{multline*}
\limsup_{r\to+\infty}l_{\nu}(r,ar)\leq \min \Bigl\{\limsup_{r\to+\infty} l_{\mu}^{\rh}(r,ar), \limsup_{r\to+\infty} l_{\mu}^{\lh}(r,ar)\Bigr\} \\
+C\limsup_{r\to+\infty}CJ_{\RR} (r,ar;q_0+q_E+q)+C,
\end{multline*}
а затем,  поделив обе части на $\ln a$ и устремив $a$ к $+\infty$, 
в обозначениях определения \eqref{{dens}barl} 
имеем 
\begin{multline*}
\ln\text{\!-}\overline\dens(l_{\nu})\leq 
\min \bigl\{\ln\text{\!-}\overline\dens( l_{\mu}^{\rh}), 
\ln\text{\!-}\overline\dens (l_{\mu}^{\lh})\bigr\}+
C\ln\text{\!-}\overline\dens J_{\RR} (\cdot,\cdot;q_0+q_E+q)\\
\overset{\eqref{lJ}}{=}\min \bigl\{\ln\text{\!-}\overline\dens( l_{\mu}^{\rh}), 
\ln\text{\!-}\overline\dens (l_{\mu}^{\lh})\bigr\}
\overset{\eqref{df:dDlLm}}{\leq}  \ln\text{\!-}\overline\dens (l_{\mu}),
\end{multline*}
что по предложению \ref{clb} о совпадении всех четырёх логарифмических блок-плотностей из  \eqref{dens} даёт в точности \eqref{D}.
\end{proof}

Для {\it нулевой ц.ф.э.т.\/} $0$ по определению множество её нулей 
$\Zero_0=\CC$, а считающая мера  
$n_{\Zero_0}(S)\overset{\eqref{nZ}}{=}+\infty$ для каждого $S\subset \CC$.  

Отметим вытекающую из теоремы  \ref{prth1}  
теорему единственности для ц.ф.э.т.: 
\begin{theorem}\label{th3}
Пусть ${\sf Z}=\{{\sf z}_k\}_{k=1,2,\dots}\subset \CC$ --- последовательность комплексных точек конечной верхней плотности в смысле \eqref{murad}, т.\,е.
\begin{equation*}
n_{\sf Z}^{\rad}(r)\overset{\eqref{nZ}}{:=}
n_{\sf Z}\bigl(D(r)\bigr)=O(r) \quad\text{при $r\to +\infty$}, 
\end{equation*} 
с логарифмической субмерой интервалов 
\begin{equation*}%%\label{Zb}
l_{\sf Z}(r,R):=
\max \left\{\sum_{\substack{r < |{\sf z}_k|\leq R\\\Re {\sf z}_k >0}} \Re \frac{1}{{\sf z}_k}, \sum_{\substack{r < |{\sf z}_k|\leq R\\\Re {\sf z}_k <0}} \Re \Bigl(-\frac{1}{{\sf z}_k}\Bigr)\right\}\overset{\eqref{df:lm}}{=}l_{n_{\sf Z}}(r,R).
\end{equation*}
Пусть  выполнены  условия \eqref{{uM}M}, \eqref{{uM}0}, \eqref{{uM}q},  \eqref{{u<M}E}, \eqref{q0}, а  ц.ф.э.т. $f$ обращается в нуль на ${\sf Z}$ в том смысле, что $n_{\Zero_f}\geq n_{\sf Z}$, и для всех $y\in \RR^+\setminus E$
 \begin{equation}\label{fM}
\ln \bigl|f(iy)f(-iy)\bigr|\leq \mathsf{C}_M\bigl(iy, q(y)\bigr) +\mathsf{C}_M\bigl(-iy, q(-y)\bigr)
+q_0(y)+q_0(-y)
\end{equation}
Если выполнено  \eqref{lJ} и при этом 
$\ln\text{\!-\!}\dens(l_{\sf Z})>\ln\text{\!-\!}\dens\bigl(J_{i\RR}(\cdot,\cdot; M)\bigr)$,
то $f=0$.
\end{theorem}
\begin{proof} Предположим, что $f\neq 0$ и положим $u:=\ln |f|$. 
Тогда имеет место \eqref{{uM}u}  с  $\nu:=n_{\sf Z}$. Из \eqref{fM} по предложению \ref{prQN}  при $a>1$ из \eqref{rRuM-} имеем
\begin{equation*}
l_{\sf Z}(r,ar)=l_{n_{\sf Z}}(r,ar)\leq   J_{i\RR}(r,ar;M)+CJ_{\RR} (r,ar;q_0+q_E+q)+C
\quad\text{при $r\geq r_0$}.
\end{equation*}
Устремляя здесь $r$ к $+\infty$, затем, после деления на $\ln a$, $a$ к $+\infty$,  по предложению \ref{pr5} и  определению \eqref{dens} ввиду \eqref{lJ} получаем 
$\ln\text{\!-\!}\overline\dens(l_{\sf Z})\leq \ln\text{\!-\!}\overline\dens\bigl(J_{i\RR}(\cdot,\cdot; M)\bigr)$,
что по предложению \ref{clb} противоречит условию $\ln\text{\!-\!}\dens(l_{\sf Z})>\ln\text{\!-\!}\dens\bigl(J_{i\RR}(\cdot,\cdot; M)\bigr)$.
\end{proof}

\renewcommand{\refname}{Литература}

\end{document}